\documentclass{amsart}
\usepackage{latexsym,amssymb,amsthm,amsmath}

\usepackage{graphics,amscd}
\usepackage{amssymb}
\usepackage{eucal}
\usepackage{amsmath}

\theoremstyle{plain}
\newtheorem{theorem}{Theorem}[section]
\newtheorem*{Theorem B}{Theorem B}
\newtheorem*{Theorem A}{Theorem A}
\newtheorem{lemma}{Lemma}[section]

\newtheorem{definition}{Definition}[section]

\numberwithin{equation}{section}

\theoremstyle{remark}
\newtheorem{remark}{Remark}[section]
 \numberwithin{equation}{section}
\newtheorem*{question}{Question}

\def\<{\left < }
\def\>{\right >}
\def\({\left ( }
\def\){\right )}

\def\e{\eqref}

\def\x{{\bf x}}

\begin{document}

\vspace{2cm}

\title[Differential geometry of rectifying submanifolds]
{Differential geometry of rectifying submanifolds}

\author{ Bang-Yen Chen}
\address{\it Michigan State University \newline\indent
Department of Mathematics  \newline\indent
619 Red Cedar Road,  \newline\indent East Lansing, Michigan
48824--1027, U.S.A.}
\email{bychen@math.msu.edu}


\date{}

\begin{abstract} A space curve in a Euclidean 3-space $\mathbb E^3$ is called a rectifying curve if its position vector field always lies in its rectifying plane. This notion of rectifying curves was introduced by the author in \cite{c6}.  
In this present article, we introduce and study the notion of rectifying submanifolds in Euclidean spaces. 
In particular, we prove that a Euclidean submanifold is rectifying if and only if the tangential component of its position vector field is a concurrent vector field.
Moreover, rectifying submanifolds with arbitrary codimension are completely determined.\end{abstract}

\maketitle

\section{Introduction}

Let $\mathbb E^3$ denote  Euclidean 3-space with its inner product $\left<\;\,,\;\right>$. Consider a  unit-speed space curve $x :  I\to \mathbb E^3$, where $I=(\alpha,\beta)$ is a real interval.
Let ${\bf x}$ denote the position vector field of $x$ and   ${\bf x}'$ be denoted by {\bf t}. 

It is possible, in general, that ${\bf t} '(s)=0$ for some $s$; however, we assume that this never happens. Then we can introduce a unique vector field ${\bf n}$ and positive function $\kappa$ so that ${\bf t}'=\kappa {\bf n}$. We call ${\bf t}'$ the {\it curvature vector field}, ${\bf n}$ the {\it principal normal vector field}, and $\kappa$ the {\it curvature} of the curve. Since ${\bf t}$ is of constant length, ${\bf n}$ is orthogonal to ${\bf t}$. The {\it binormal vector field} is defined by ${\bf b}={\bf t}\times{\bf n}$, which is a unit vector field  orthogonal to both ${\bf t}$ and ${\bf n}$. One defines the {\it torsion} $\tau$ by the equation ${\bf b}'=-\tau{\bf n}$. 

The famous Frenet-Serret equations are given by  
\begin{align}\label{E:S-F} &\begin{cases}{\bf t}'=\hskip.38in\kappa {\bf n}\\ {\bf n}'=-\kappa {\bf t} \hskip.3in+\tau{\bf b}  
\\{\bf b}'=\hskip.3in-\tau{\bf n}.\end{cases}\end{align}
At each point of the curve, the planes spanned by $\{{\bf t},{\bf n}\}$, $\{{\bf t},{\bf b}\}$, and $\{{\bf n},{\bf b}\}$
  are known as the {\it osculating plane}, the {\it rectifying plane}, and the {\it normal plane}, respectively. 

From elementary differential geometry it is well known that a curve in $\mathbb E^3$ lies in a plane if its position vector lies  in its osculating plane at each point, and lies on a sphere if its position vector lies in its normal plane at each point. 
In view of these basic facts, the author asked the following simple geometric question in \cite{c6}:

\begin{question} When does the position vector of a space curve $\hbox{\bf x}:  I\to \mathbb E^3$ always lie in its rectifying plane?
\end{question}

The author called such a curve  a {\it rectifying curve} in \cite{c6}. The author derived many fundamental properties of rectifying curves. In particular, he completely classifies all rectifying curves in \cite{c6}. It is known that rectifying curves related with the notions constant-ratio curves and convolution (cf. \cite{c2001,c2002a,c2002b,c2003a,c2003b}). Furthermore, the author and F. Dillen established in \cite{cd05} a  simple link between rectifying curves and the notion of centrodes in mechanics. Moreover, they showed in \cite{cd05} that rectifying curves are indeed the extremal
curves which satisfy the equality case of a general inequality. 
Since then rectifying curves have been studied by many authors, see \cite{Cam16,II2003,II2007,II2008,II2014,Lucas2015,O2009,Yi16,Yu2007} among many others. For the most recent survey on rectifying curves, see \cite{c16}.

In this article, we extend the notion of rectifying curves to the notion of rectifying submanifolds in a very natural way. Many fundamental properties of rectifying submanifolds are obtained. In particular, we prove that a Euclidean submanifold is rectifying if and only if the tangential component of its position vector field is a concurrent vector field.
Moreover, rectifying submanifolds with arbitrary codimension are completely determined.

\section{Preliminaries}

Let $x: M\to \mathbb E^m$ be an isometric immersion of a Riemannian manifold $M$ into the Euclidean $m$-space $\mathbb E^m$. For each point $p\in M$, we denote by
$T_pM$ and $T^\perp_p M$ the tangent and the normal spaces
at $p$. 

There is a natural orthogonal decomposition:
\begin{equation}\label{2.1} T_p{\mathbb E}^{m}=T_pM\oplus T^\perp_p
M.\end{equation}

Denote by $\nabla$ and $\tilde\nabla$ the Levi-Civita connections of $M$ and ${\mathbb E}^{m}$, respectively. The formulas of Gauss and Weingarten  are given respectively by (cf. \cite{cbook,book11})
\begin{align}\label{2.2}  &\tilde \nabla_XY=\nabla_X Y+h(X,Y),
\\& \label{2.3}   \tilde\nabla_X\xi =-A_\xi X+D_X\xi \end{align}  
for vector fields $X,\,Y$ tangent to  $M$ and  $\xi$ normal to $M$, where  $h$ is the second fundamental form, $D$ the normal connection, and $A$ the  shape operator of $M$. 

For a given point $p\in M$, the {\it first normal space}, of $M$ in $\mathbb E^m$, denoted by ${\rm Im}\,h_p$, is the subspace defined by
\begin{align} \label{2.4} {\rm Im}\,h_p ={\rm Span}\{h(X,Y):X,Y\in T_p M\}.\end{align}

For each normal vector $\xi$ at $p$, the shape operator $A_\xi$ is a self-adjoint endomorphism of $T_pM$. 
The second fundamental form $h$ and the shape operator $A$ are related by
\begin{equation}\label{2.5}  \<A_\xi X,Y\>=\<h(X,Y), \xi\>,\end{equation}  where $\<\;\, ,\;\>$
is the inner product on $M$ as well as on the ambient Euclidean space.

The  {\it equation of Gauss\/} of $M$ in $\mathbb E^m$ is given by
\begin{align}\label{2.6} R(X,Y;Z,W)=\<\sigma(X,W),\sigma(Y,Z)\>-\<\sigma(X,Z),\sigma(Y,W)\> \end{align}  for
$X,Y,Z,W$ tangent to $M$, where $R$ denotes the curvature tensors of $M$.

The covariant derivative ${\bar \nabla}h$ of $h$ with respect to the connection on
$TM \oplus T^{\perp}M$ is defined by
\begin{align}\label{2.7}({\bar\nabla}_{X}h)(Y,Z)=D_{X}(h (Y,Z))-h(\nabla_{X}Y,Z)-h(Y,\nabla_{X}Z).\end{align} 

The {\it equation of Codazzi\/} is
\begin{align}\label{2.8}({\bar\nabla}_{X}h)(Y,Z)= ({\bar\nabla}_{Y}h)(X,Z).\end{align} 

It follows from the definition of a rectifying curve $x:I\to \mathbb E^3$ that the position vector field $\x$ of $x$ satisfies
\begin{equation}\label{2.9} {\bf x}(s)=\lambda(s){\bf t}(s)+\mu(s){\bf b}(s)\end{equation}
for some functions $\lambda$ and $\mu$.  

For a curve  $x:I\to \mathbb E^3$ with $\kappa(s_0)\ne 0$ at  $s_0\in I$, the first normal space at $s_0$ is the line spanned by the principal normal vector ${\bf n}(s_0)$. 
Hence, the rectifying plane at $s_0$ is nothing but the plane orthogonal to the first normal space at $s_0$. Therefore, for a submanifold $M$ of $\mathbb E^m$ and  a point $p\in M$, we call the subspace of $T_p\mathbb E^m$, orthogonal complement to the first normal space ${\rm Im}\,\sigma_p$, the {\it rectifying space of} $M$ at $p$.

\begin{definition}\label{D:2.1} {\rm A submanifold $M$ of a Euclidean $m$-space $\mathbb E^m$ is called a} {\it rectifying submanifold} {\rm if the position vector field $\x$ of $M$ always lies in its rectifying space. In other words, $M$ is called a rectifying submanifold if and only if 
\begin{equation}\label{2.10} \<\x(p),{\rm Im}\, h_p\>=0\end{equation}
holds at every $p\in M$.}
\end{definition}

\begin{definition} {\rm A non-trivial vector field $Z$ on a Riemannian manifold $M$ is called a {\it concurrent vector field} if it satisfies
\begin{align}\label{2.11}\nabla_X Z=X\end{align}
for any vector  $X$ tangent to $M$, where $\nabla$ is the Levi-Civita connection of $M$.}
\end{definition}

\section{Lemmas}

 By a {\it cone\/} in $\mathbb E^m$ with vertex at the origin we mean a ruled submanifold generated by a family of lines  passing through the origin. A submanifold of $\mathbb E^m$ is called a {\it conic
submanifold\/} with vertex at the origin if it is an open portion of a cone with vertex at the origin. 

There exists a natural orthogonal decomposition of the position vector field $\x$ at each point for a Euclidean submanifold $M$; namely,
\begin{equation} \x=\x^T+\x^N,\end{equation} where $\x^T$ and
$\x^N$ denote the tangential and normal components of $\x$, respectively. Let $|\x^T|$ and $|\x^N|$
be the length of $\x^T$ and $\x^N$, respectively.

\begin{lemma} \label{L:3.1} Let $\,x:M\to {\mathbb E}^{m}$ be an isometric immersion of a Riemannian
$n$-manifold into the Euclidean $m$-space ${\mathbb E}^{m}$. Then $\x=\x^T$ holds identically if and only if  $M$ is a conic submanifold with the vertex at the origin. \end{lemma}
\begin{proof} Let $\,x:M\to {\mathbb E}^{m}$ be an isometric immersion of a Riemannian
$n$-manifold into the Euclidean $m$-space ${\mathbb E}^{m}$. If $\x=\x^T$ holds identically, then
$e_1=\x/|\x|$ is a unit vector field tangent  to $M$. 

Put $\x=\rho e_1$. Since $\tilde\nabla_{e_1}e_1$ is perpendicular to $e_1$, we find from
\begin{equation}\tilde\nabla_{e_1}\x=e_1,\;\; \tilde\nabla_{e_1}\x=(e_1\rho)e_1+\rho
\tilde\nabla_{e_1}e_1,\end{equation} 
that $\tilde\nabla_{e_1}e_1=0$. Therefore, the integral curves of $e_1$ are some open portions of generating lines in $\mathbb E^m$. Moreover,  because $\x=\x^T$,  the generating lines  given by the integral curves of $e_1$  pass through the origin. Consequently, $M$ is a conic submanifold with the vertex at the origin.

The converse is clear. \end{proof}

\begin{lemma} \label{L:3.2} Let $\,x:M\to {\mathbb E}^{m}$ be an isometric immersion of a Riemannian
$n$-manifold into the Euclidean $m$-space ${\mathbb E}^{m}$. Then $\x=\x^N$ holds identically if and only if  $M$ lies in a hypersphere centered at the origin. \end{lemma}
\begin{proof}  Let $\,x:M\to {\mathbb E}^{m}$ be an isometric immersion of a Riemannian
$n$-manifold into the Euclidean $m$-space ${\mathbb E}^{m}$. If $\x=\x^N$ holds identically, then  we get $$Z\! \<\x,\x\>=2\<\right.\! \tilde\nabla_Z\x,\x\left.\!\>=2\<Z,\x^N\>=0$$ for any $Z\in TM$. Thus $M$ lies in a hypersphere centered at the origin. 

The converse is obvious.
\end{proof}

In views of Lemma \ref{L:3.1} and Lemma \ref{L:3.2} we make the following.

\begin{definition} {\rm A rectifying submanifold $M$ of $\mathbb E^m$ is called } {\it proper} {\rm if its position vector field $\x$ satisfies $\x\ne \x^T$ and $\x\ne \x^N$ at every point on $M$.}
\end{definition}

\begin{lemma} \label{L:3.3}  Let $M$ be a proper rectifying submanifold of $\mathbb E^m$ with $\dim M=n$. Then we have
\begin{align}\label{3.5} m> n+\dim\, ({\rm Im}\,h_p)\end{align} for each $p\in M$. \end{lemma}
\begin{proof}  Let $M$ be a proper rectifying submanifold of $\mathbb E^m$. If $m=n+\dim\, ({\rm Im}\,h_p)$, then we get $\x=\x^T$ which is a contradiction.
\end{proof}

\begin{remark} In views of Lemma \ref{L:3.1} and Lemma \ref{L:3.2}, we are only interested on proper rectifying submanifolds.\end{remark}

\section{Characterization and classification of  rectifying submanifolds}

First, we give the following simple characterization of rectifying submanifolds.

\begin{theorem} \label{T:4.1}  If the position vector field $\x$ of a submanifold $M$  in ${\mathbb E}^{m}$ satisfies $\x^N\ne 0$, then $M$ is a proper rectifying submanifold  if and only if $\x^T$ is concurrent vector field on $M$.
\end{theorem}
\begin{proof}
Let  $x:M\to {\mathbb E}^{m}$ be an isometric immersion of a Riemannian $n$-manifold into the Euclidean $m$-space ${\mathbb E}^{m}$. Consider the orthogonal decomposition 
\begin{align}\label{4.1}  \x=\x^T +\x^N\end{align}
of the position vector field $\x$ of $M$ in $\mathbb E^m$.  

From \e{4.1} and formulas of Gauss and Weingarten,  we find
\begin{align}\label{4.2}  Z=\tilde\nabla_Z \x=\nabla_Z \x^T +h(Z,\x^T)-A_{\x^N}Z +D_Z\x^N\end{align}
for any $Z\in TM$. After comparing the tangential components in \e{4.2}, we obtain
\begin{align}\label{4.3}  A_{\x^N}Z=\nabla_Z \x^T - Z.\end{align}

Assume that $M$ is a proper rectifying submanifold. Then we have $\x^T\ne 0$ and $\x^N\ne 0$ . Moreover, it follows from the Definition \ref{D:2.1} that
\begin{align}\label{4.4}  \<\x, h(X,Y)\>=0\end{align}
for $X,Y\in TM$. So we get $A_{\x^N}=0$. Hence, we obtain from \e{4.3} that 
\begin{align}\label{4.5}  \nabla_Z \x^T = Z,\end{align}
which shows that $\x^T$ is a concurrent vector field on $M$.

Conversely, if $\x^T$ is a concurrent vector field on $M$, then we find from \e{2.11} and \e{4.3} that $A_{\x^N}=0$. Therefore we obtain \e{4.4}.  Consequently, $M$ is a proper rectifying submanifold due to $\x^N\ne 0$ by assumption.
\end{proof}
 
Next, we give the following classification of rectifying submanifolds.

\begin{theorem}\label{T:4.2} If $M$ is a proper rectifying submanifold of ${\mathbb E}^{m}$, then with respect to some suitable local coordinate systems $\{s,u_2,\ldots,u_n\}$ on $M$  the immersion $\,x$ of $M$ in $\mathbb E^m$  is of the form
\begin{equation}\label{4.6} x(s,u_2,\ldots,u_n)=\sqrt{s^2+c^2}\,
Y(s,u_2,\ldots,u_n),\;\; \<Y,Y\>=1,\; c>0,\end{equation} such that the metric tensor $g_Y\! $ of the spherical submanifold defined by $Y$ satisfies
\begin{align}\label{4.7}  g_Y=\frac{c^2}{(s^2+c^2)^2}ds^2+\frac{s^2}{s^2+c^2}\sum_{i,j=2}^n g_{ij}(u_2,\ldots,u_n) du_i du_j.\end{align}

Conversely,  the immersion given by \e{4.6}-\e{4.7} defines a proper rectifying submanifold.
\end{theorem}
\begin{proof} Let  $x:M\to {\mathbb E}^{m}$ be an isometric immersion of a Riemannian $n$-manifold $M$ into the Euclidean $m$-space ${\mathbb E}^{m}$. Assume that $M$ is a proper rectifying submanifold.  Then \e{4.2} holds.

After comparing  the normal components of \e{4.2}, we obtain
\begin{align}\label{4.8} &D_Z\x^N=-h(Z,\x^T),\end{align}
for $Z\in TM$.

It follows from \e{4.4} and \e{4.8} that $\<\x,D_Z \x^N\>=0$. Hence we get $$Z\! \<\x^N,\x^N\>=0,$$ which implies that $\x^N$ is of positive constant length, say $c$.  
 From \e{4.4} we  obtain
\begin{equation}\label{4.11} \<A_{\x^N} X,Y\>=\<\x^N,h(X,Y)\>=\<\x,h(X,Y)\>=0.\end{equation}
Hence we have $A_{\x^N}=0$. 
Let us put $\rho=|\x^T|$ and $e_1=\x^T/\rho$. We may extend $e_1$ to a local orthonormal frame $e_1,\ldots,e_n$. 

We put
\begin{equation}\label{4.9} \nabla_X e_i=\sum_{j=1}^n \omega_i^j(X)e_j,\;\; i=1,\ldots,n.\end{equation} 
For  $j,k=2,\ldots,n$, we find
\begin{equation}\label{4.10} 0=e_k\! \<\x,e_j\>=\delta_{jk}+ \<\x,\nabla_{e_k}e_j\>+\<\x,h(e_j,e_k)\>,\end{equation} 
Since $h(e_j,e_k)=h(e_k,e_j)$, equation \e{4.10} gives 
 $$\omega^1_j(e_k)=\omega^1_k(e_j),\;\; j,k=2,\ldots,n.$$ 
 Hence, it follows from the Frobenius theorem that the distribution $\mathcal D$
spanned by $e_2,\ldots,e_n$ is an integrable distribution. 

On the other hand, the distribution $\mathcal D^\perp={\rm Span}\,\{e_1\}$ is also integrable since it is of rank one. Therefore, there exist local coordinate systems $\{s,u_2,\ldots,u_n\}$ on $M$ such that $e_1=\partial/\partial s$ and $\partial/\partial u_2,\ldots,\partial/\partial u_n$ span the distribution $\mathcal D$.

 Let us put 
\begin{equation}\label{4.12}\x^T=\varphi e_1\end{equation} with $\varphi=|\x^T|$.  
By taking the derivative of $\varphi=\<\x,e_1\>$ with respect to $e_j$ for $j=1,\ldots,n$, we also have
\begin{equation}\label{4.13}e_j\varphi=\delta_{1j}+\<\x,h(e_1,e_j)\>.\end{equation}
Combining \e{4.4} and \e{4.13} gives
\begin{equation}\label{4.14}e_j\varphi=\delta_{1j},\quad j=1,\ldots,n.\end{equation} 
Therefore, we obtain $\varphi=\varphi(s)$ and $\varphi'(s)=1$ which imply
$\varphi(s)=s+b$ for some constant $b$. Thus, after applying a suitable translation on $s$ if necessary,
we have $\varphi=s$. Consequently, the position
vector field satisfies  \begin{equation}\label{4.15}\x=se_1+\x^N.\end{equation} 
By combining \e{4.15} and $|\x^N|=c$, we find 
\begin{equation}\label{4.16}\<\x,\x\>= s^2+c^2, \end{equation}
where $c$ is a positive number. Hence we may put
\begin{equation}\label{4.17} x(s,u_2,\ldots,u_n)=\sqrt{s^2+c^2}\,
Y(s,u_2,\ldots,u_n),\end{equation} for some $\mathbb E^m$-valued function
$Y=Y(s,u_2,\ldots,u_n)$ satisfying $\<Y,Y\>=1$. 

Using \e{4.17} and the fact that $e_1=\partial/\partial s$ is orthogonal to the distribution $\mathcal D$, we obtain that
\begin{equation}\label{4.18} \<Y_s,Y_s\>={c^2\over {(s^2+c^2)^2}},\;\; \<Y_s,Y_{u_j}\>=0, \;\; j=2,\ldots,n.\end{equation}
Therefore, the metric tensor $g_Y$ of the spherical submanifold defined by $Y$ takes the following form:
\begin{align}\label{4.19}  g_Y=\frac{c^2}{(s^2+c^2)^2}ds^2+\sum_{i,j=2}^n g_{ij}(s,u_2,\ldots,u_n) du_i du_j.\end{align}

On the other hand, it follows from Theorem \ref{T:4.1} that $\x^T=se_1$ is a concurrent vector field. Thus, we find from \e{4.5} that 
\begin{equation}\begin{aligned}\label{4.20} e_1= \nabla_{e_1}\x^T=\nabla_{e_1} se_1=e_1+s\nabla_{e_1}e_1.\end{aligned}\end{equation} Hence we get $\nabla_{e_1}e_1=0$, which implies that the integral curves of $e_1$ are geodesic in $M$. Therefore, the distribution $\mathcal D^\perp$ spanned by $e_1$ is a totally geodesic foliation. 

From \e{4.5} we have
\begin{equation} e_i= \nabla_{e_i} \x^T=s\nabla_{e_i}e_1, \;\; i=2,\ldots,n,\end{equation}
which implies that
\begin{equation} \label{4.22}\omega^j_1(e_i)=\frac{\delta_{ij}}{s}, \;\; i,j=2,\ldots,n,\end{equation}
where $\delta_{ij}=1$ or 0 depending on $i=j$ or $i\ne j$. 

From \e{4.22} we conclude that $\mathcal D$ is an integrable distribution whose leaves are totally umbilical  in $M$. Moreover, the mean curvature of leaves of $\mathcal D$  are given by $s^{-1}$. Since the leaves of $\mathcal D$ are hypersurfaces in $M$, it follows that the mean curvature vector fields of leaves of $\mathcal D_2$ are parallel in the normal bundle of $M$ in $\mathbb E^m$. Therefore, $\mathcal D$ is a spherical foliation. Consequently, by a result of \cite{H} (or Theorem 4.4 of \cite[page 90]{book11}) we conclude that $M$ is locally a warped product $I\times_{s}  F$, where  $F$ is a Riemannian $(n-1)$-manifold. Thus, the metric tensor $g$ of $M$ takes the form 
\begin{equation} \label{4.23}g=ds^2+s^2 g_F,\end{equation}
where $g_F$ is the metric tensor of $F$.
Now, by applying \e{4.6}, \e{4.19} and \e{4.23}, we may conclude that  the metric tensor $g_Y$  can  be expressed as \e{4.7}.

Conversely, let us consider a submanifold $M$ of $\mathbb E^m$ defined by 
\begin{equation}\label{4.24} x(s,u_2,\ldots,u_n)=\sqrt{s^2+c^2}\,
Y(s,u_2,\ldots,u_n),\;\; \<Y,Y\>=1,\; c>0,\end{equation} such that the metric tensor $g_Y\! $  satisfies
\begin{align}\label{4.25}  g_Y=\frac{c^2}{(s^2+c^2)^2}ds^2+\frac{s^2}{s^2+c^2}\sum_{i,j=2}^n g_{ij}(u_2,\ldots,u_n) du_i du_j.\end{align} 
Then it follows from \e{4.24} that
\begin{equation}\begin{aligned}\label{4.26} &\frac{\partial \x}{\partial s}=\frac{sY}{\sqrt{s^2+c^2}}+\sqrt{s^2+c^2}\, Y_s,\;\; 
\\&\frac{\partial \x}{\partial u_j}=\sqrt{s^2+c^2}\,
Y_{u_j},\;\; j=2,\ldots,n, \end{aligned}\end{equation}
where $Y_s=\partial Y/\partial s$ and $Y_{u_j}=\partial Y/\partial u_j$.
It follows from \e{4.24}, \e{4.25} and \e{4.26} that the metric tensor $g_M$ of $M$ is given by
\begin{equation} \label{4.27}g_M=ds^2+s^2 \sum_{i,j=2}^n g_{ij}(u_2,\ldots,u_n) du_i du_j. \end{equation}

Now, by an easy computation, we find from \e{4.27} that
\begin{equation} \label{4.28}\nabla_{\frac{\partial}{\partial s}}\frac{\partial}{\partial s}=0, \;\; \nabla_{\frac{\partial}{\partial u_j}} \frac{\partial}{\partial s}=\frac{1}{s}\frac{\partial}{\partial u_j},\;\; j=2,\ldots,n.\end{equation}

Since $\<Y,Y\>=1$,  \e{4.24} and \e{4.26} imply that
\begin{equation} \label{4.29} \<\right.\! \x, \x_{u_j}\! \left.\>=0, \;\;  j=2,\ldots,n.\end{equation}
Therefore, we obtain $\x^T=s \frac{\partial}{\partial s}$. Now, it is easy to verify that $\x^T$ is a concurrent vector field on $M$. Moreover, it is direct to show that the normal component of $\x$ is given by
$$\x^N=\frac{c^2}{\sqrt{s^2+c^2}}\,Y-s\sqrt{s^2+c^2}\, Y_s, $$
which is alway non-zero everywhere on $M$. 
Consequently, $M$ is a proper rectifying submanifold, according to Theorem \ref{T:4.1}.
\end{proof}

\begin{remark} Theorem \ref{T:4.2} extend  Theorem 3 of \cite{c6}.
\end{remark}

\begin{remark} If we put $s=\tan^{-1}\(\frac{t}{c}\)$, then \e{4.7} becomes 
\begin{align}\label{4.30} g_Y=dt^2+\sin^2t \sum_{j,k=2}^n g_{jk}(u_2,\ldots,u_n)du_j du_k.\end{align}
For $n=2$, we get $g_Y=dt^2+(\sin^2t) du^2$ from \e{4.30}, which is the metric tensor of a spherical coordinate system $(t,u)$ on $S^2(1)$.  Hence, for $n=2$, $Y=Y(t,u)$ is nothing but an isometric immersion from an open portion of  $S^1(1)$ into $S^{m-1}(1)\subset \mathbb E^m$.
Therefore, there exist many spherical submanifolds in $\mathbb E^m$ whose metric tensor is given by \e{4.7}. Consequently, there exist many  rectifying submanifolds in $\mathbb E^m$ according to Theorem \ref{T:4.2}.
\end{remark}

\section{Some properties of rectifying submanifolds}

Finally, we provide some basic properties of proper rectifying submanifolds.

\begin{theorem}\label{T:5.1} Let  $M$ be a proper rectifying submanifold of ${\mathbb E}^{m}$.  Then 
\begin{itemize}
\item[(a)] $|\x^T|=s+b$ for some constant $b$.

\item[{(b)}] $|\x|^2=s^2+c_1 s+c_2$ for some constants $c_1$ and $c_2$.

\item[(c)] $\x^N$ is of constant length.

\item[(d)] $A_{\x^N}=0$.

\item[(e)]  The curvature tensor $R$ satisfies $R(\x^T,Y)=0$ for any $Y \in TM$.

\item[(f)] The sectional curvature $K$ of $M$ satisfies $K(\x^T,Z)=0$ for any unit vector $Z$ perpendicular to $\x^T$.
\end{itemize}
\end{theorem}
\begin{proof} Statements (a), (b), (c) and (d)  are already done in the proof of Theorem \ref{T:4.2}.
Clearly, statement (f) follows immediately from statement (e).

Now, we prove statement (e). This can be done as follows. By applying \e{2.6} and \e{4.8} we have
\begin{equation}\begin{aligned}\label{5.1} R(\x^T,Y,Z;W)\, &=\<h(\x^T,W),h(Y,Z)\>-\<h(\x^T,Z),h(Y,W)\>
\\&=\<D_Z\x^N,h(Y,W)\>-\<D_W \x^N,h(Y,Z)\>
\\&=-\<\x^N,D_Z h(Y,W)\>+\<\x^N,D_W h(Y,Z)\>.\end{aligned}\end{equation}
Therefore, after applying \e{4.4} and equation \e{2.8} of Codazzi, we derive from \e{5.1} that
\begin{equation}\begin{aligned} R(\x^T,Y,Z;W)\, &=\<\x^N,(\bar\nabla_W h)(Y,Z)\>-\<\x^N,(\bar\nabla_Z h)(Y,W)\>=0,\end{aligned}\end{equation}
which gives statement (e). 
\end{proof}

\begin{remark} Statement (a), (b) and (c) of Theorem \ref{T:5.1} extend the corresponding results obtained in Theorem 1 of \cite{c6}.
\end{remark}

\begin{remark} One may define rectifying submanifolds in a pseudo-Euclidean space in the same as Definition \ref{D:2.1}.  We will treat rectifying submanifolds in pseudo-Euclidean spaces in a separate article.
\end{remark}

\end{document}